\documentclass[a4paper,9pt]{article}
\RequirePackage[T1]{fontenc}
\RequirePackage{fix-cm}
\usepackage{amsmath,amsfonts,amsthm}
\usepackage{graphicx} 
\usepackage[unicode, final, colorlinks=true]{hyperref} 
\usepackage{cite}
\usepackage{caption}
\usepackage{mathtools}
\usepackage{newtxtext}
\usepackage[varvw]{newtxmath}
\usepackage{bm,bbm}
\usepackage[T1]{fontenc}
\usepackage[utf8]{inputenc}
\usepackage[dvipsnames]{xcolor}
\usepackage{tikz, pgfplots}
\usepackage{tikzscale}
\usepackage{booktabs}
\usepackage{multirow}
\usepackage{array}
\usepackage{csquotes}
\usepackage{indentfirst}
\usepackage{enumitem}
\usepackage{accents}

\hypersetup{citecolor=teal, linkcolor=BrickRed, urlcolor=NavyBlue}

\newcommand{\ddt}{\partial_t}
\newcommand{\ddx}{\partial_x}
\newcommand{\R}{\mathbb{R}}

\newtheorem{assumption}{Assumption}

\newtheorem{theorem}{Theorem}[section]
\newtheorem{lemma}[theorem]{Lemma}

\newtheorem{remark}[theorem]{Remark}

\providecommand{\keywords}[1]{\textit{Keywords:} #1}
\providecommand{\msc}[1]{\textit{2020 MSC:} #1}

\pgfplotsset{compat=1.18}
\usepgfplotslibrary{groupplots}
\usepgfplotslibrary{fillbetween}
\usetikzlibrary{backgrounds}
\usetikzlibrary{arrows.meta}
\pgfplotsset{plot coordinates/math parser=false}
\usepgfplotslibrary{patchplots}
\usetikzlibrary{arrows.meta,bending,chains,decorations.markings}
\begin{document}
\title{Influx ratio preserving coupling conditions for the networked Lighthill--Whitham--Richards model}
\author{
  Niklas~Kolbe
}

\date{
  \small
  Institute of Geometry and Applied Mathematics,\\ RWTH Aachen University, Templergraben 55,\\ 52062 Aachen, Germany\\
   \smallskip
   {\tt kolbe@igpm.rwth-aachen.de} \\
   \smallskip
   \today
 }

\maketitle                   

\abstract{
A new coupling rule for the Lighthill–Whitham–Richards model at merging junctions is introduced that imposes the preservation of the ratio between inflow from a given road to the total inflow into the junction. This rule is considered both in the context of the original traffic flow model and a relaxation setting giving rise to two different Riemann solvers that are discussed for merging 2-to-1 junctions. Numerical experiments are shown suggesting that the relaxation based Riemann solver is capable of suitable predictions of both free-flow and congestion scenarios without relying on flow maximization.  
}

\msc{35L65, 90B10, 90B20}

\keywords{Macroscopic traffic flow models, coupling conditions, hyperbolic conservation laws}

\section{Introduction}
Traffic flow has been described by fluid dynamics-like macroscopic models since the 1950s. The modeling of road networks by means of hyperbolic equations coupled on a graph has only been an active research field in recent years, see e.g.~\cite{bressan2014flowsnetwor, coclite2005traffflowroadnetwor}. The modeling of the junctions requires that a certain set of rules is fixed, such that both, the coupling problem is well posed and typical driver behavior is reflected~\cite{garavello2006traffflownetwor}. Recently, also data-driven approaches to the modeling of junctions have been considered~\cite{herty2022datadrivenmodel}. Moreover, relaxation, see e.g. \cite{borsche2018}, and kinetic approaches \cite{borsche2018kinet} have been used to derive suitable coupling conditions. In this work we consider a new coupling condition for merging junction and embed it in a networked traffic model and a relaxation model.

We model road networks using a directed graph with nodes representing the junctions and edges representing the roads. On each edge the Lighthill-–Whitham–-Richards (LWR) model
\begin{equation}\label{eq:lwr}
  \ddt \rho^k + \ddx \left(\rho^k V^k(\rho^k) \right) = 0\ \qquad \text{in } (0, \infty) \times \mathcal{E}_k 
\end{equation}
proposed in \cite{lighthill1955kinemwavesii, richards1956shockwaveshighw} determines the time and space evolution of the vehicle density $\rho$. In \eqref{eq:lwr} we indicate the corresponding edge by $k$ and assume a parametrization in physical space over the interval $\mathcal{E}_k\subset \R$. For simplicity, we limit this study to networks consisting of a single node/junction. This allows us to partition all edges within a set of incoming and a set of outgoing edges and hence to use the notations $\delta^- \sqcup \delta^+ = \delta^\mp = \{1, 2, \dots, m \}$. On each road $k$ the velocity $V^k$ is a function of the vehicle density, a relation referred to as fundamental diagram. It is typically chosen such that the flux $f^k(\rho^k) \coloneqq \rho^k V^k(\rho^k)$ is a strictly concave function in $\rho^k$. A classical choice is due to \cite{greenshields1935} and assumes the linear relation
\begin{equation}\label{eq:greenshields}
V^k(\rho^k) = v^k_\text{max} \left(1 - \frac{\rho^k}{\rho^k_\text{max}} \right)
\end{equation}
for suitable parameters $\rho^k_\text{max}$ and $v^k_\text{max}$ representing the stagnation density and the maximal velocity, respectively. For a better fit to measurement data on the roads various more detailed fundamental diagrams that require further parameters have been proposed, see e.g.,~\cite{newell1993ii}.

\section{Coupling conditions for the networked Lighthill–-Whitham-–Richards model}\label{sec:cc}
In this section we recall common coupling conditions for the network model \eqref{eq:lwr} at the node and propose a new one for merging roads.

\paragraph{Mass conservation} We assume initial data $\rho_0^k$ for $k \in \delta^\mp$ being given. To define weak solutions we introduce the smooth test functions $\Phi_k: (0, \infty) \times \mathcal{E}_k\to \R$ for all roads $k \in \delta^\mp$ that are smooth across the junction, i.e., it holds $\Phi_k(\cdot, b_k) = \Phi_\ell(\cdot, a_\ell)$ for $k\in \delta^-$ and $\ell \in \delta^+$ assuming that the road parametrizations take the form $\mathcal{E}_k=(a_k,b_k)$. A weak solution $(\rho^k)_{k \in \delta^\mp}$ then must satisfy
\begin{equation}\label{eq:kirchhoff}
  \int_0^T \int_{\mathcal{E}_k} \rho^k \ddt \Phi_k + \rho^k V^k(\rho^k) \, dx \, dt = \int_{\mathcal{E}_k} \rho^k_0 \Phi^k(0, x) \, dx.
\end{equation}
This weak formulation implies that the mass needs to be conserved in the junction, see \cite{holden1995}. The resulting condition is known as \emph{Kirchhoff conditions} and in general reads
\begin{equation}
  \sum_{k\in \delta^-} f^k(\rho^k(t, b_k^-)) = \sum_{\ell \in \delta^+} f^\ell(\rho^\ell(t, a_\ell^+)) \qquad \text{for a.~e. }t>0
\end{equation}
at the junction. Here the terms $b_k^-$ and $a_\ell^+$ refer to the limit from the left and the right as $x$ approaches the corresponding boundary of the road.

\paragraph{Outgoing roads} To determine the traffic on the outgoing roads, a traffic distribution matrix with entries $\{ a_{\ell k}\}_{\ell \in \delta^+, k \in \delta^-}$ is fixed, which satisfy the property
\[
  \sum_{\ell \in \delta^+} \alpha_{\ell k} = 1 \qquad \forall k \in \delta^-,
\]
see e.g., \cite{coclite2005traffflowroadnetwor}. The entry $a_{\ell k}$ determines the ratio of the inflow from road $k$ that is distributed to the outgoing road $\ell$ and consequently the nodal conditions
\begin{equation}\label{eq:outgoing}
  q_{\ell k}(t)  = \alpha_{\ell k} f^k(\rho^k(t, b_k^-)) \qquad \text{for a.~e. }t>0 \text{ and } \forall k \in \delta^-, \ell \in \delta^+,
\end{equation}
where $q_{\ell k}$ denotes the flow from road $k$ into road $\ell$ so that
\[
  \sum_{k \in \delta^-} q_{\ell,k}(t) = f^{\ell}(\rho^\ell(t, a_\ell^+)) \qquad \forall \ell \in \delta^-, \qquad \sum_{\ell \in \delta^+} q_{\ell,k}(t) = f^{k}(\rho^k(t, b_k^+)) \quad  \forall k \in \delta^+.
\]
We note that for well-posedness of the coupling problem the distribution matrix needs to satisfy some technical conditions, see \cite{garavello2006traffflownetwor} for details.

\paragraph{Incoming roads}
Similar to outgoing roads, the states on the outgoing roads can be determined by imposing right of way parameters to which the outgoing fluxes are to be proportional \cite{gottlich2021seconordertraff, garavello2006traffflownetwor}. In this work we propose another rule that considers the ratio between the influx from a selected road to the total influx, i.e.,
\[
  r^k(t, x) = \frac{ f^k(\rho^k(t, x))}{ \sum_{k\in \delta^-} f^k(\rho^k(t, x)) }
\]
Assuming a positive total influx $\sum_{k\in \delta^-} f^k(\rho^k(t, b_k))$ into the junction we stipulate a vanishing first spatial derivative of this influx ratio at the interface:
\begin{equation}\label{eq:influxratiopreservation}
  \ddx r^k(t, b_k) = 0 \quad \forall k \in \delta^-. 
\end{equation}
As this condition reflects a preservation of the influx ratio at the junction it can be compared to dynamic right of way parameters that are inherited from the incoming streets. Note that, as argued in \cite{herty2023centr}, condition \eqref{eq:influxratiopreservation} needs to be imposed for only $|\delta^-| - 1$ roads; the condition for the final road then follows from a summation argument.

Since we only allow for nonnegative flows continuity of the zero-influxes is imposed in the case that the total influx is zero.

\paragraph{Demand and supply conditions}
To obtain admissible entropy solutions of the conservation law \eqref{eq:lwr} a half Riemann problem needs to be solved on each road. In case of the networked LWR model this is equivalent to the \emph{demand and supply conditions} proposed in \cite{lebacque1996godun} that read
\begin{equation}\label{eq:ds}
  \begin{aligned}
    0 \leq f^k(\rho^k(t, b_k^-)) &\leq d^k(\rho^k(t, b_k^-)) \quad \forall k \in \delta^-, \\
    0 \leq f^\ell(\rho^\ell(t, a_\ell^+)) &\leq s^\ell(\rho^\ell(t, a_\ell^+)) \quad \forall \ell \in \delta^+
  \end{aligned}
  \qquad \text{for a.~e. }t>0.
\end{equation}
Denoting on all roads the unique maximizer of $V^k$ by $\sigma^k$ demand and supply are given by the functions
\[
  d^k(\rho^k) = \begin{cases*}
    \rho^k V^k(\rho^k) & if $\rho^k\leq \sigma^k$\\
    \sigma^k V^k(\sigma^k) & if $\rho^k> \sigma^k$
  \end{cases*}
  \quad \forall k \in \delta^-, \qquad
  s^\ell(\rho^\ell) = \begin{cases*}
    \sigma^\ell V^\ell(\sigma^\ell) & if $\rho^\ell\leq \sigma^\ell$ \\
    \rho^\ell V^\ell(\rho^\ell) & if $\rho^\ell> \sigma^\ell$
  \end{cases*}
  \quad \forall k \in \delta^-.
\]
Condition \eqref{eq:ds} gives rise to an admissible region of solutions. Out of all admissible solution the one that maximizes the total flow in the junction is selected, i.e., the maximizer of
\begin{equation}\label{eq:maxflow}
\max   \sum_{k\in \delta^-} f^k(\rho^k(t, b_k^-)) 
\end{equation}
constrained by the nodal conditions in this section.
\begin{remark}
The flow maximization \eqref{eq:maxflow} as well as the demand and supply conditions \eqref{eq:ds} are not required for the relaxation based coupling introduced in the next section. From a numerical point of view this reduces the complexity of the nodal solver as no optimization problem needs to be solved.
\end{remark}

\section{A relaxation based Riemann solver}\label{sec:relax}
Following the approach in \cite{herty2023centrschemtwo} we formulate the Jin-Xin relaxation system \cite{jin1995relaxschemsystem} for the networked model \eqref{eq:lwr} and obtain 
\begin{equation}\label{eq:relaxation}
  \begin{aligned}
    \ddt \rho^k_{\varepsilon} + \ddx v_{\varepsilon}^k &= 0, \\
    \ddt v_{\varepsilon}^k + \lambda^k \ddx \rho^k_{\varepsilon} &= \frac{1}{\varepsilon}(v_{\varepsilon}^k - \rho^k_\varepsilon V^k(\rho^k_{\varepsilon})).
  \end{aligned}
\end{equation}
In this system the flux of the original system is replaced by the new variable $v_\varepsilon$, which in turn is given by a balance law that involves the relaxation rate $\varepsilon$ and the relaxation speed $\lambda^k$. In the limit $\varepsilon \to 0$ the original system \eqref{eq:lwr} is recovered. We follow the approach in \cite{herty2023centrschemtwo} and construct a nodal solver for \eqref{eq:relaxation}, which after taking the system to the relaxation limit serves as a nodal solver for the original system \eqref{eq:lwr}.  

A Riemann solver at the junction is a map of the form
\begin{equation}\label{eq:rs}
\mathcal{RS}: (\rho_o^1, v_0^1, \dots, \rho_o^m, v_0^m) \mapsto (\rho_R^1, v_R^1, \dots, \rho_L^m, v_L^m), \qquad m = |\delta^-| + |\delta^+|
\end{equation}
assigning trace data near the junction from a numerical solution to coupling data that serves as Dirichlet boundary data in the following time step. The coupling data is chosen such that it solves $m$ half Riemann problems with respect to system \eqref{eq:relaxation} and the trace data. Due to the simple structure of the relaxation system it follows that on all roads the trace and coupling states are connected via the linear relations
\begin{equation}
  \rho^k_R =\rho^k_0 - \sigma^k, \quad v^k_R = v^k_0+ \sigma^k \lambda \quad \forall k \in \delta^-, \qquad  \rho^\ell_L =\rho^\ell_0 + \sigma^\ell, \quad v^k_R = v^k_0+ \sigma^k \lambda \quad \forall \ell \in \delta^+
\end{equation}
for some parameters $\sigma^k, \sigma^\ell \in \R$. Thus there are $m$ degrees of freedom that are determined by adopting the coupling conditions from Section \ref{sec:cc} as follows.

We impose the Kirchhoff conditions \eqref{eq:kirchhoff} for the relaxation states $\rho^k_\varepsilon$ and note that the consistency with the networked LWR model, see \cite{herty2023centrschemtwo}, implies a similar condition for the states $v^k$. Together, the following conditions are set.
\begin{align}
  \sum_{k\in \delta^-} \rho^k_R V^k(\rho^k_R) &= \sum_{\ell \in \delta^+} \rho^\ell_L V^\ell(\rho^\ell_L), \label{eq:relkirchhoffu}\\
   \sum_{k\in \delta^-} v^k_R &= \sum_{\ell \in \delta^+} v_L^\ell.\label{eq:relkirchhoffv}
\end{align}
In addition, we impose condition \eqref{eq:outgoing} for all $\ell \in \delta^+$ and condition \eqref{eq:influxratiopreservation} for all $k \in \delta^-$. As the states $v_\varepsilon^k$ represent the fluxes in the relaxation limit these conditions are imposed substituting
\begin{equation}
  f^\ell(\rho^\ell(t, a_\ell^+)) = v_L^\ell \quad \forall \ell \in \delta^+, \qquad  f^k(\rho^k(t, b_k^+)) = v_R^k \quad \forall k \in \delta^-. 
\end{equation}
It is easy to verify that in total $m$ independent conditions are obtained. Due to \eqref{eq:relkirchhoffu} an overall nonlinear system in the parameters $\sigma^1, \dots, \sigma^m$ is obtained. If this system has a unique solution in $\R^m$ this solution determines the coupling data of \eqref{eq:rs}. In case of multiple solutions in $\R^m$ the one that minimizes the distance from the trace data, i.e. $|\sum_{k\in \delta^{\mp}}(\sigma^k)^2|$ is selected. By its design the new Riemann solver clearly is consistent as it satisfies
\[
  \mathcal{RS}(\rho_R^1, v_R^1, \dots, \rho_L^m, v_L^m) = (\rho_R^1, v_R^1, \dots, \rho_L^m, v_L^m)
\]
if coupling data is already given by $\rho_R^1, v_R^1, \dots, \rho_L^m, v_L^m$.

\section{Merging 2-to-1 junctions}
In this section we consider the case of a junction with two incoming and a single outgoing road as shown in Fig.~\ref{fig:2to1}. We assume that on each of the roads the velocity is given by a fundamental diagram of the form \eqref{eq:greenshields}. We study an application of the Riemann solver derived in Section~\ref{sec:relax} to this case in Section \ref{sec:2to1relax} and also provide a Riemann solver for the original networked LWR model in Section \ref{sec:2to1rs}.

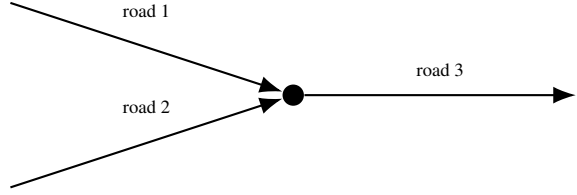
\begin{figure}[h]
  \centering
\begin{tikzpicture}[x=\linewidth/47,y=\linewidth/39] \scriptsize 
  \node[shape=circle, fill=black, inner sep=0pt, minimum size=8pt] (N) at (0,0) {};
  \draw[-{Latex[length=3mm]}, thick] (-11,3) -- node[above=.3cm] {road 1} ++ (N);
  \draw[-{Latex[length=3mm]}, thick] (-11,-3) -- node[above=.3cm] {road 2} ++ (N);
  \draw[-{Latex[length=3mm]}, thick] (N) -- node[above=.15cm] {road 3} ++ (11, 0);
\end{tikzpicture}
\caption{Merging junction with two incoming and one outgoing road.}\label{fig:2to1}
\end{figure}
\subsection{The relaxation based Riemann solver at merging junctions}\label{sec:2to1relax}
The Riemann solver \eqref{eq:rs} at the merging junction in Fig.~\ref{fig:2to1} assigns the states $\rho^1_R$, $v^1_R$, $\rho^2_R$, $v^2_R$, $\rho^3_L$ and $v_L^3$ to the data $\rho^1_0$, $v_0^1$, $\rho^2_0$, $v^2_0$, $\rho^3_0$ and $v_0^3$. Let $G_0 \coloneqq \rho^1_0 V^1(\rho^1_0) + \rho^1_0 V^2(\rho^2_0) - \rho^3_0 V^3(\rho^3_0)$ and $G_1 \coloneqq v_0^1 + v_0^2 - v_0^3$ then the conditions \eqref{eq:relkirchhoffu} and \eqref{eq:relkirchhoffv} recast as
\begin{align}
  G_0 + d_1(\sigma^1) + d_2(\sigma^2) &= d_3(\sigma^3), \label{eq:2to1relu}\\
 \lambda \left( \sigma^3 - \sigma^1 - \sigma^2 \right) & = G_1, \label{eq:2to1relv}
\end{align}
where in \eqref{eq:2to1relu} we employ the expressions
\[
  d_k(\sigma^k) = (\rho_0^k-\sigma^k)V^k(\rho_0^k-\sigma^k) - \rho_0^k V^k(\rho_0^k)\quad k \in \{1,2\}, \qquad d_3(\sigma^3) = (\rho_0^3+\sigma^3)V^3(\rho_0^3+\sigma^3) - \rho_0^3 V^3(\rho_0^3).
\]
In addition, we obtain from the influx ratio preservation \eqref{eq:influxratiopreservation} the relation
\begin{equation}\label{eq:2to1influxpreservation}
\frac{v_0^2}{v_0^1 + v_0^2} \sigma^1 = \frac{v_0^1}{v_0^1 + v_0^2} \sigma^2
\end{equation}
assuming $v_0^1 + v_0^2 >0$. Equations \eqref{eq:2to1influxpreservation} and \eqref{eq:2to1relv} allow us to write both $\sigma^1$ and $\sigma^2$ as linear functions of $\sigma^3 \eqqcolon \sigma$ and the system reduces to the nonlinear equation
\begin{equation}\label{eq:2to1final}
  G_0 + d_1\left( \frac{r_1}{\lambda}( \lambda \sigma - G_1)\right) + d_2\left( \frac{r_2}{\lambda}( \lambda \sigma - G_1) \right) = d_3(\sigma), \qquad r_1 =  \frac{v_0^1}{v_0^1 + v_0^2}, \quad r_2 =  \frac{v_0^2}{v_0^1 + v_0^2}.
\end{equation}

The following results discuss the solvability of \eqref{eq:2to1final}. If solutions exists the one with minimal absolute value determines the coupling data of the Riemann solver.
\begin{lemma}\label{lemma}
  Suppose that $v_0^1 + v_0^2>0$ as well as the bounds
  \begin{align}
 r_1^2 \frac{G_1^2 \, v^1_{max}}{\lambda^2 \, \rho^1_{max}} + r_2^2 \frac{G_1^2 \, v^2_{max}}{\lambda^2 \, \rho^2_{max}}  &\leq G_0 + \frac{G_1}{\lambda} \left(r_1 (f^1)^\prime(\rho^1) + r_2 (f^2)^\prime(\rho^2) \right), \label{eq:lemc1}\\
\frac{v^3_{max}}{\rho^3_{max}} &\leq  \frac{v^1_{max}}{\rho^1_{max}} \, r_1^2  + \frac{v^2_{max}}{\rho^2_{max}} \, r_2^2   \label{eq:lemc2}
   \end{align}
   hold then \eqref{eq:2to1final} has at least one real solution. The same consequence follows if the converse inequalities of both, \eqref{eq:lemc1} and \eqref{eq:lemc2} hold.
 \end{lemma}
 \begin{proof}
   If we write $d_1\left( \frac{r_1}{\lambda}( \lambda \sigma - G_1)\right)$, $ d_2\left( \frac{r_2}{\lambda}( \lambda \sigma - G_1) \right)$ and $-d_3(\sigma)$ as polynomials in $\sigma$ with coefficients $d_{k0}$, $d_{k1}$ and $d_{k2}$ for $k\in\{1,2,3\}$ we observe that the discriminant of \eqref{eq:2to1final} takes the form
   \begin{equation}\label{eq:discriminant}
     -4(G_0 +d_{10} + d_{20} + d_{30}) (d_{12} + d_{22} + d_{32}) + (d_{11} + d_{21} + d_{31})^2.
   \end{equation}
   A computation of the coefficients then shows that \eqref{eq:lemc1} ensures that $G_0 +d_{10} + d_{20} + d_{30}$ is non-negative and \eqref{eq:lemc2} ensures that $d_{12} + d_{22} + d_{32}$ is non-positive. Hence, the conditions imply non-negativity of the discriminant. 
 \end{proof}
 \begin{remark}
   The signs of $G_0$ and $G_1$ are not known a priori. In the context of the relaxed scheme introduced in the next section it holds $G_0=G_1$, which allows for better control over sign of the right hand side in \eqref{eq:lemc1}. Even in cases, in which the assumptions of Lemma~\ref{lemma} are not satisfied an analysis of \eqref{eq:discriminant} shows that at least one real solution exists if
   \[
     r_1 (f^1)^\prime(\rho^1_0) + r_2 (f^2)^\prime(\rho^2_0) + (f^2)^\prime(\rho^2_0)
   \]
   does not vanish and both, $|G_0|$ and $|G_1|$ are sufficiently small.
 \end{remark}
 \subsection{An influx ratio preserving Riemann solver for the unrelaxed model}\label{sec:2to1rs}
 In this section we discuss a Riemann solver for \eqref{eq:lwr} on the junction in Fig.~\ref{fig:2to1}. As in this setting demand and supply conditions need to be imposed to obtain an admissible entropy solution we make the following assumptions about the prioritization of the coupling conditions.
 \begin{assumption} Flow maximization \eqref{eq:maxflow} under demand and supply conditions \eqref{eq:ds} is prioritized over the ratio preserving condition \eqref{eq:influxratiopreservation}.
 \end{assumption}
 For simplicity we consider here a Riemann solver mapping the traces at the junction to coupling fluxes at the interface, i.e.,
 \begin{equation}\label{eq:rs2}
\mathcal{RS}^*: (\rho_o^1, \rho_o^2, \rho_o^3) \mapsto (f_R^1, f_R^2, f_L^3).
\end{equation}
Indeed, this formulation is equivalent to the more common one, which maps to coupling states, cf.~\cite{herty2022datadrivenmodel}.

Given the states $\rho_0^1$, $\rho_0^2$ and $\rho_0^3$ the coupling fluxes are determined by the following procedure:
There are two cases that can occur.
In the case of free flow, where $ d_1(\rho_0^1) + d_2(\rho_0^2) \leq s_3(\rho_0^3)$, we take $f^1_R = d_1(\rho_0^1)$, $f^2_R = d_2(\rho_0^2)$ and $f^3_L= d_1(\rho_0^1) + d_2(\rho_0^2)$. This choice satisfies the demand and supply condition and maximizes the flow in the junction.
In the complimentary congestion case we set $f^3_L=s_3(u_0^3)$ and employ the influx ratios to set $f_R^1=r_1 s_3(u_0^3)$ and $f_R^2=r_2 s_3(u_0^3)$. If this leads to a violation of \eqref{eq:ds} for either $f^1_R$ or $f^2_R$ the affected coupling flux is chosen as the respective upper bound and its counterpart is computed from the Kirchhoff condition.

\section{Numerical experiments}
In this section we consider numerical experiments for the LWR model on the merging junction shown in Fig~\ref{fig:2to1}. In particular, we employ the influx ratio preserving coupling condition and compare the Riemann solver based on the entropy admissible coupling condition for the nonlinear system \eqref{eq:rs2} to the relaxation based Riemann solver \eqref{eq:rs}.

We employ a variant of the scheme from \cite{herty2023centr} that has been derived by taking an asymptotic preserving numerical method for the networked relaxation system \eqref{eq:relaxation} to its relaxation limit. This way the scheme naturally offers a discretization of the networked LWR model employing the relaxation based Riemann solver. The roads and their parametrizations that for simplicity we assume to be $(-1, 0)$ on the incoming and $(0, 1)$ on the outgoing road are discretized over $M$ equidistant cells of length $\Delta x$. By $\rho_j^{k,n}$ we denote an approximate average of $\rho^k$ over the cell $I_j=[(j-1/2)\Delta x, (j+1/2)\Delta x]$ at time instance $t_n$. In addition we introduce the time increment $\Delta t$ satisfying the CFL condition  $\Delta t = \text{CFL}\, \frac{\Delta x}{\lambda}$. The scheme then has the conservative form 
\begin{equation}\label{eq:conservativeformnet}
  \rho_j^{k,n+1} = \rho_j^{k,n} - \frac{\Delta t }{\Delta x} \left( F_{j+1/2}^{k,n} - F_{j-1/2}^{k,n}\right).
\end{equation}
The index $j$ in~\eqref{eq:conservativeformnet} is taken $j= -M,\dots, -1$ for $k=1,2$ or $j= 0, 1, \dots, M$ for $k=3$. The numerical fluxes are given by
\begin{equation}\label{eq:netfluxes}
  F_{j-1/2}^{k,n} =
  \begin{cases}
    \frac 1 2 \, (f^k(\rho_{j}^{k,n}) + f^k(\rho_{j-1}^{k,n}))  - \frac {\lambda} 2 (\rho_j^{k,n}-\rho_{j-1}^{k,n}) & \text{if }j \neq 0\\[5pt]
    f_R^{k,n} &\text{if }j=0\text{ and }k\in\{1,2\}\\[5pt]
    f_L^{3,n} &\text{if }j=0\text{ and }k=3 
  \end{cases},
\end{equation}
where the coupling fluxes are obtained applying the Riemann solver~\eqref{eq:rs} to the numerical trace states:
\[
  (\rho_R^{1,n}, f^{1,n}_R, \rho_R^{2, n}, f^{1,n}_R, \rho_L^{3,n}, f^{3,n}_R) = \mathcal{RS}(\rho^{1,n}_{-1}, f^1(\rho^{1,n}_{-1}), \rho^{2,n}_{-1},  f^2(\rho^{2,n}_{-1}), \rho^{3,n}_0,  f^3(\rho^{3,n}_{0})).
\]
Alternatively, the nonlinear Riemann solver introduced in Section~\ref{sec:2to1rs} can be used in the scheme by taking
\[
   (f^{1,n}_R,  f^{1,n}_R, f^{3,n}_R) = \mathcal{RS}^*(\rho^{1,n}_{-1},  \rho^{2,n}_{-1}, \rho^{3,n}_0).
 \]
In our numerical experiments we employ $M=1000$ mesh cells per road.
 
 \begin{figure}[ht]
   \centering
   \includegraphics[width=.8\linewidth]{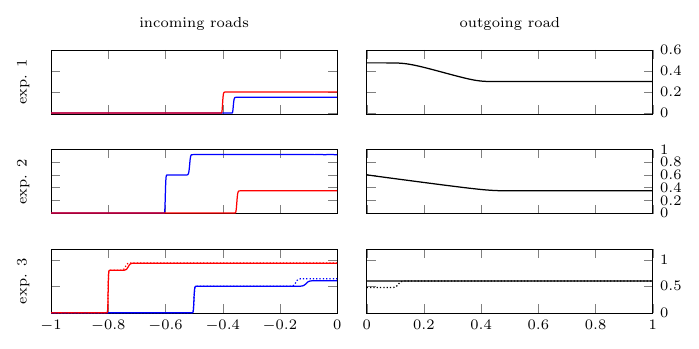}
   \caption{Numerical solutions to Experiments 1--3 showing the incoming vehicle densities $\rho^1$ (blue lines) and $\rho^2$ (red lines) on the left and the outgoing vehicle density $\rho^3$ on the right at the final time ($T=0.75$ in Experiment~1 and $T=1$ in Experiments 2 and 3). In Experiment~3 the solution obtained by the relaxation based Riemann solver is shown using dotted lines.}\label{fig:experiments}
 \end{figure}

 We consider three numerical experiments with results shown in Fig.~\ref{fig:experiments}. The basic setting is adopted from \cite{herty2023centr}: two roads of the same size merge into a third one with slightly higher capacity, which is reflected in the velocity functions
 \[
   V^1(\rho) = V^2(\rho) = 1- \rho, \qquad V^3(\rho) = 1 - \frac{\rho}{1.2}. 
   \]
   We use Riemann initial data and impose outgoing boundary condition, i.e., zero-flux conditions for the incoming roads and homogeneous Neumann boundary conditions for the outgoing road. In Experiment~1 we set $\rho^{1,0} \equiv 0.15$, $\rho^{2,0} \equiv 0.2$ and $\rho^{3,0} \equiv 0.3$. The setup results in free flow of the vehicles through the junction as the outgoing road has sufficient capacity for the incoming traffic.  This behavior is recovered by our numerical simulations using both Riemann solvers with results shown in the first row of Figure~\ref{fig:experiments}. In Experiment~2 we initially set $\rho^{1,0} \equiv 0.6$, $\rho^{2,0} \equiv 0.35$ and $\rho^{3,0} \equiv 0.35$, which causes congestion in the junction. In the numerical results we see a backwards propagating congestion wave forming on road 1, while traffic from road 2 freely passes through the junction. In both, Experiments 1 and 2 both Riemann solvers yield the same numerical results; the relative difference of the corresponding numerical results at the final time has reached machine accuracy. We emphasize that the relaxation approach yields the expected dynamics without flow maximization.

For Experiment~3 we consider another congestion scenario, for which we choose  $\rho^{1,0} \equiv 0.5$, $\rho^{2,0} \equiv 0.8$ and $\rho^{3,0} \equiv 0.6$. Note that those initial densities yield maximal flow on road 1 and road 3. In the numerical results we see congestion waves forming on both, road 1 and road 2. In this experiment our two Riemann solvers yield different results. In particular, the relaxation based Riemann solver predicts a small drop in vehicle densities on road 3 close to the junction and thus does not maximize the flow through the junction. The new Riemann solver thus introduces inefficiencies in traffic flow in case of complex congestion cases at the junction as can be also expected in real traffic. We note that throughout our numerical simulation the vehicle densities have stayed within $[0, \rho^k_{max}]$ at the respective roads and no boundary layers have been obtained.  

%
\bibliographystyle{abbrvurl}
\bibliography{references.bib}
\end{document}